\documentclass[journal,dvipsnames]{IEEEtran}

\usepackage{pgfplots}

\usepackage{amsmath} 
\usepackage{amssymb}  
\usepackage[bottom]{footmisc}
\usepackage{graphicx}
\usepackage{epstopdf}
\usepackage{mathtools}
\usepackage{dsfont}
\usepackage{url}
\usepackage{color}
\usepackage{booktabs} 
\usepackage{multirow}   
\usepackage{balance}    
\usepackage{soul}   
\usepackage{hyperref}   
\usepackage{tikz}
\usepackage{siunitx}
\usepackage{xcolor}
\usepackage{cellspace}
\usepackage{hyperref}

\DeclareMathOperator*{\argmax}{argmax}

\usepackage{balance}    
\usepackage{amsthm}
\def\sq{\mathbin{{\strut\rule{1.25ex}{1.25ex}}}}

\usepackage{balance}    
\usepackage{amsthm}
\usepackage{xcolor}
\usepackage{tcolorbox}
\usepackage{dsfont}
\usepackage[noend]{algorithm}

\makeatletter
\newcommand{\removelatexerror}{\let\@latex@error\@gobble}
\renewcommand{\fnum@figure}{Figure \thefigure}
\makeatother

\newcommand{\beq}{\begin{equation}}
\newcommand{\eeq}{\end{equation}}

\usepackage{bm}

\newcounter{algorithmctr}[section]
\renewcommand{\thealgorithmctr}{\thesection.\arabic{algorithmctr}}
{\refstepcounter{algorithmctr}\begin{list}{}{%
\setlength{\rightmargin}{0\linewidth}%
\setlength{\leftmargin}{.05\linewidth}
\setlength{\itemsep}{1pt}
\setlength{\parskip}{0pt}
\setlength{\parsep}{0pt}}%
\rmfamily\small
\item[]{\setlength{\parskip}{0ex}\hrulefill\par%
\nopagebreak{\bfseries\textsf{Algorithm \thealgorithmctr~}}}}%
{{\setlength{\parskip}{-1ex}\nopagebreak\par\hrulefill} \end{list}}
\IEEEoverridecommandlockouts

\usepackage[noadjust]{cite} 
\usepackage{amsmath,stmaryrd,graphicx}

\newtheoremstyle{boldStyle}
  {\topsep}
  {\topsep}
  {\itshape}
  {0pt}
  {\bfseries}
  {.}
  { }
  {\thmname{#1}\thmnumber{ #2}\thmnote{ (#3)}}

\newtheoremstyle{italicStyle}
  {\topsep}
  {\topsep}
  {}
  {0pt}
  {\bfseries}
  {.}
  { }
  {\thmname{#1}\thmnumber{ #2}\thmnote{ (#3)}}

\theoremstyle{boldStyle}

\newtheorem{theorem}{Theorem}
\newtheorem{lemma}{Lemma}
\newtheorem{proposition}{Proposition}

\newtheorem{corollary}{Corollary}

\theoremstyle{italicStyle}
\newtheorem{assumption}{Assumption}
\newtheorem{remark}{Remark}

\renewenvironment{proof}{{\textbf{Proof:}}}{\hfill$\sq$}

\newcommand{\FinTAc}{F^{\mathrm{in},\mathrm{active}}_{t:t+T-1}}
\newcommand{\GeqT}{G^{\mathrm{eq}}_{t:t+T-1}}
\newcommand{\FinT}{F^{\mathrm{in}}_{t:t+T-1}}
\newcommand{\GeqNm}{G^{\mathrm{eq}}_{t:t+N-2}}
\newcommand{\GeqN}{G^{\mathrm{eq}}_{t:t+N-1}}
\newcommand{\FinN}{F^{\mathrm{in}}_{t:t+N-1}}

\newcommand{\FinNAcm}{F^{\mathrm{in},\mathrm{active}}_{t:t+N-2}}
\newcommand{\MatT}{M_{t:t+N-1}}
\newcommand{\MatTm}{M_{t:t+N-2}}
\newcommand{\MatTmp}{M_{t+1:t+N-1}}
\newcommand{\LagEq}{\boldsymbol{\lambda}^*}
\newcommand{\LagIn}{\boldsymbol{\delta}^{*,\textrm{active}}}
\usepackage[noend]{algpseudocode}
\usepackage{dsfont}
\usepackage{algorithm}

\usepackage{xcolor}
\usepackage{tcolorbox}

\definecolor{ugoColor}{rgb}{0.6,0.8,0.0}
\definecolor{ligthGray}{rgb}{0.95,0.95,0.95}

\usepackage{hyperref}

\definecolor{color1}{HTML}{003f5c}
\definecolor{color2}{HTML}{7a5195}
\definecolor{color3}{HTML}{ef5675}
\definecolor{color4}{HTML}{ffa600}

\pgfplotscreateplotcyclelist{foo}{
        {color1,mark=*,mark options={fill=white}, thick},
        {color4,mark=square*,mark options={fill=white}, thick},
        {color2,mark=diamond*,mark options={fill=white}, thick},
        {color3,mark=triangle*,mark options={fill=white}, thick},
      }

 \pgfplotsset{mytikzstyle/.style={
     every axis/.append style={
         legend style={
             draw = none,
             },
         },
        ylabel near ticks,         
        axis x line=bottom,
        axis y line=left,
        enlargelimits=false,
        ymajorgrids=true,
        width=\columnwidth,
        cycle list name=foo,
     },}
 \pgfplotsset{mytikzstyle}

\title{\LARGE \bf On the Optimality and Convergence Properties of the \\ Iterative Learning Model Predictive Controller  }

\author{Ugo Rosolia, Yingzhao Lian, Emilio T. Maddalena, Giancarlo Ferrari-Trecate, Colin N. Jones
\thanks{This work has received support from the Swiss National Science Foundation under the RISK project (Risk Aware Data-Driven Demand Response), grant number 200021 175627. Ugo Rosolia is with Caltech, USA.
Yingzhao Lian, Emilio T. Maddalena, Giancarlo Ferrari-Trecate, Colin N. Jones are with the Automatic Control Laboratory, École polytechnique fédérale de Lausanne (EPFL) E-mails:\tt\scriptsize{\{ugo.rosolia\}@gmail.com.}, \tt\scriptsize{\{yingzhao.lian, emilio.maddalena, giancarlo.ferraritrecate, colin.jones\}@epfl.ch.}}%
}

\begin{document}

\maketitle
\thispagestyle{empty}
\pagestyle{empty}
\begin{abstract}
In this technical note we analyse the performance improvement and optimality properties of the Learning Model Predictive Control (LMPC) strategy for linear deterministic systems. The LMPC framework is a policy iteration scheme where closed-loop trajectories are used to update the control policy for the next execution of the control task. We show that, when a Linear Independence Constraint Qualification (LICQ) condition holds, the LMPC scheme guarantees strict iterative performance improvement and optimality, meaning that the closed-loop cost evaluated over the entire task converges asymptotically to the optimal cost of the  infinite-horizon control problem. Compared to previous works this sufficient LICQ condition can be easily checked, it holds for a larger class of systems and it can be used to adaptively select the prediction horizon of the controller, as demonstrated by a numerical example.
\end{abstract}
 
\section{Introduction}

Model Predicitve Control (MPC) is an established control methodology which systematically uses forecast to compute control actions~\cite{scokaert1998constrained, rawlings2000tutorial, grieder2004computation, borrelli2017predictive}. In MPC at each time step, a \textit{model} is used to predict the evolution of the system over a time horizon. A sequence of control actions is chosen such that the predicted trajectory safely drives the system from the current measured state to a \textit{safe set} and it minimizes the predicted cost over the prediction horizon and the future cost given by a \textit{value function}. The MPC policy applies the first predicted input to the system, and the process is repeated at the next time step based on the new measurement. 
Computing the safe set and the value function to maximize the region of attraction of the controller and to guarantee optimal closed-loop performance is hard.
To guarantee these properties the safe set should be the \textit{maximal stabilizable set} (i.e., a set containing all states from which the control task can be executed~\cite{mayne2000constrained, borrelli2017predictive})
and the value function should map each state of the safe set to the cumulative cost associated with the \textit{optimal policy}~\cite{mayne2000constrained}.

Several strategies have been proposed to approximate the safe set and value function used for MPC design. In the control community, these approximation strategies are based on interpolation techniques which leverage closed-loop trajectories and knowledge of the system dynamics~\cite{bacic2003general,scibilia2011feasible, limon2005enlarging, brunner2013stabilizing, blanchini2003relatively,korda2016controller}. The resulting controllers guarantee constraint satisfaction and closed-loop stability. The approximation of the optimal value function is studied in Approximate Dynamic Programming (ADP) and Reinforcement Learning (RL). In ADP and RL the value function estimate is computed either fitting Monte Carlo estimate of the cumulative cost or minimizing the temporal difference associated with stored historical data~\cite{bertsekas1996neuro,bertsekas2019feature,sutton1988learning,bradtke1996linear}.

In this note, we focus on the learning model predictive control (LMPC) technique presented in~\cite{LMPC} and specialized for linear dynamics in~\cite{LMPC_linear}. In this framework, 
the control task is performed iteratively and the terminal safe set and value function are constructed using historical data collected from previous iterations of the control task. The asymptotic properties of this iterative procedure were studied in the original work~\cite{LMPC}, where the authors established that the closed-loop cost is non-increasing as iterations progress. Moreover, it was shown that if the algorithm attains a fixed-point, and a specific set containment condition is met, then the closed-loop trajectory at convergence is the optimizer of a constrained problem with an arbitrarily long but finite horizon. 

Our contributions in this note are twofold.  First, we show that when the Linear Independence Constraint Qualification (LICQ) condition holds and the closed-loop costs of two subsequent iterations are equal, then the closed-loop trajectory is optimal for the \textit{infinite-horizon} optimal control problem. Second, we leverage this result to guarantee \textit{strict} performance improvement at each iteration of the control task and, therefore, convergence of the LMPC algorithm to a fixed-point which is optimal for the infinite-horizon optimal control problem.
From a practical viewpoint, the proposed LICQ condition is simple to verify after the algorithm has converged and it can be employed by users to guide the selection of the LMPC prediction horizon.

\textit{Notation:} 
A polyhedron $P$ is an intersection of a finite number of half-spaces and a ball of radius $\epsilon$ is denoted as $\mathcal{B}(\epsilon) = \{x \in \mathbb{R}^2  ~ | ~ ||x||_2 \leq \epsilon \} $. As usual in calculus, infinite sums are to be interpreted in the limit sense, i.e., $\sum_{k=0}^{\infty} f(k) \coloneqq \lim_{n \rightarrow \infty} \sum_{k=0}^{n} f(k)$. The convex hull $\text{Conv}(\cdot)$ of a countable set of points is the smallest closed convex set containing them. The identity matrix is represented by $I_N \in \mathbb{R}^{N\times N}$ and the vector of ones, by $\mathds{1}_N \in \mathbb{R}^N$. A matrix $M \succ (\succeq) 0$ is positive (semi) definite and $\otimes$ denotes the Kronecker product. 
Finally, given $m$ vectors $v_1, \ldots, v_m$ in $\mathbb{R}^n$ we define the vector $\text{Vec}(v_1, \ldots, v_m) = [v_1^\top, \ldots, v_m^\top]^\top \in \mathbb{R}^{nm}$. 

\section{Problem Formulation}
Consider the linear deterministic system
\begin{equation}\label{eq:sys}
    x_{t+1} = A x_t + B u_t,
\end{equation}
where at time $t$ the state $x_t \in \mathbb{R}^n$ and the input $u_t \in \mathbb{R}^d$. The system is subject to the following state and input constraints:
\begin{equation}\label{eq:constraints}
\begin{aligned}
    x_t \in \mathcal{X}&=\{x\in \mathbb{R}^n: F_x x\leq b_x\}, \\ u_t \in \mathcal{U}&=\{u\in \mathbb{R}^d: F_u u\leq b_u\},
\end{aligned}
\end{equation}
which are assumed to be polyhedra containing the origin in their interior and should be satisfied for all $t\geq0$.

Our goal is to find a control policy $\pi^*(\cdot)$ which solves the following infinite-time optimal control problem:
\begin{equation}\label{eq:inf_OCP}
\begin{aligned}
        J_{0 \rightarrow \infty}^*(x_S) = \min_{u_0,u_1,\ldots} \quad & \sum_{t=0}^\infty h(x_t, u_t) \\
    \text{subject to} ~~\quad & x_{t+1} = A x_t + B u_t,~\forall t \geq 0, \\
    &x_t \in \mathcal{X}, u_t \in \mathcal{U},~\forall t \geq 0, \\
    &x_0 = x_S,
\end{aligned}
\end{equation}
where $x_S$ is a known initial condition and the running cost $h(x, u)$ fulfills the following assumption.
\begin{assumption}\label{ass:costStrictCovexity}
The function $h(x, u) = x^\top Q x + u^\top R u$ for the matrices $Q \succeq 0$ and $R \succ 0$.
\end{assumption}

\section{Learning Model Predictive Control}
In this section, we recall the iterative Learning Model Predictive Control (LMPC) strategy~\cite{LMPC,LMPC_linear}. We approximate the solution to problem~\eqref{eq:inf_OCP} by iteratively performing the regulation task from the initial condition $x_S$. 
At each iteration of the control task $j$, we store closed-loop trajectories and input sequences 
\begin{equation}\label{eq:cl_data}
    \boldsymbol{x}^j = [x_0^j, x_1^j, 
    \ldots] \text{ and } \boldsymbol{u}^j = [u_0^j, u_1^j, 
    \ldots],
\end{equation}
where $x_t^j \in \mathbb{R}^n$ and $u_t^j \in \mathbb{R}^d$ are the state and input of system~\eqref{eq:sys} at time $t$ of the $j$-th iteration. 
In the following we show that the above trajectories can be used iteratively to synthesize a predictive control policy. 

\subsection{Safe Set and Value Function Approximation}

Given $j$ stored closed-loop trajectories over an infinite horizon, we define the convex safe set at iteration $j$ as
\begin{equation}\label{eq:CS}
    \mathcal{CS}^j = \text{Conv} \Big( \bigcup_{i = 0}^j \bigcup_{t = 0}^\infty x_t^i \Big).
\end{equation}
The above convex safe set represents the convex hull of the stored trajectories and it is a control invariant set, if the stored trajectories converge to the origin~\cite{LMPC_linear}.

Next, we construct 
an approximation to the value function over the convex safe set. First, we compute the cost-to-go associated with the stored state $x_t^i$ as 
\begin{equation}\label{eq:realizedCost}
    J_{t\rightarrow \infty}^i (x_t^i)  = \sum_{k=t}^\infty h(x_k^i, u_k^i).
\end{equation}
Finally, we define the following $V$-function:
\begin{equation}\label{eq:Qfun}
\begin{aligned}
    V^j(x) = \min_{\gamma_t^i \geq 0 } \quad & \sum_{i=0}^j \sum_{k=0}^\infty \gamma_k^i J_{k\rightarrow \infty}^i (x_k^i) \\
\text{subject to} \quad & \sum_{i=0}^j \sum_{k=0}^\infty \gamma_k^i x_k^i = x,~\sum_{i=0}^j \sum_{k=0}^\infty \gamma_k^i = 1, 
\end{aligned}
\end{equation}
which interpolates the cost-to-go over the convex safe~\eqref{eq:CS}. As shown in~\cite{LMPC_linear}, the above $V$-function is a control Lyapunov function for the linear deterministic system~\eqref{eq:sys}. 

\begin{assumption}\label{ass:stored} At iteration $j=0$, we are given the closed-loop trajectory ${\boldsymbol{x}}^0= [x_0^0,\ldots, x_t^0, \ldots]$ and associated input sequence ${\boldsymbol{u}}^0= [u_0^0,\ldots, u_t^0, \ldots]$. Moreover, the state $x_t^0 \in \mathcal{X}$ and the input $u_t^0 \in \mathcal{U}$ for all $t\geq0$. Finally, $\lim_{t\rightarrow \infty} x_t^0 = 0$, $\lim_{t\rightarrow \infty} u_t^0 = 0$ and $J^0_{0\rightarrow\infty}(x_s) = \sum_{t=0}^\infty h(x_t^0,u_t^0) < \infty$.
\end{assumption}

\subsection{LMPC Policy}
In this section, we describe the LMPC design, where the $V$-function and the convex safe set are used respectively as terminal cost and terminal constraint. In particular, at time $t$ we solve the following finite-time optimal control problem:
\begin{equation}\label{eq:FTOCP}
    \begin{aligned}
        J_{t\rightarrow t+N}^{\scalebox{0.4}{LMPC},j}(x_t^j)=\quad\quad\quad&\\
        \min_{u_{t|t},\ldots,u_{t+N-1|t}}\quad & \sum_{k=t}^{t+N-1}  h(x_{k|t},u_{k|t}) + V^{j-1}(x_{t+N|t}) \\
\text{subject to}\quad  \quad \quad & x_{k+1|t}=Ax_{k|t}+Bu_{k|t}, \\
   & x_{k|t} \in \mathcal{X}, ~ u_{k|t} \in \mathcal{U}, \\
   &x_{t+N|t} \in ~\mathcal{CS}^{j-1}, \\
   &x_{t|t}=x_t^j,\\
   &\forall k \in \{t, \ldots, t+N-1\},
    \end{aligned}
\end{equation}
whose solution steers the system from the current state $x_t^j$ to the convex safe set $\mathcal{CS}^{j-1}$ constructed using the stored $j-1$ trajectories. Let $[u^{j,*}_{t|t},\ldots,u^{j,*}_{t+N-1|t}]$ be the optimal solution to problem~\eqref{eq:FTOCP}, we apply to system~\eqref{eq:sys} the following input:
\begin{equation}\label{eq:policy}
\pi^j(x_t^j) = u^{j,*}_{t|t}.
\end{equation}
The finite-time optimal control problem~\eqref{eq:FTOCP} is solved at time $t+1$, based on the new state $x_{t+1|t+1} = x_{t+1}^j$, yielding a \textit{moving} or \textit{receding horizon} control strategy.

\subsection{Properties}
In this section we recall the closed-loop properties of the iterative LMPC strategy from~\cite{LMPC_linear}.

\begin{proposition}[Theorem~1 in \cite{LMPC_linear}]\label{prop:recFeasibility}
Consider system~\eqref{eq:sys}~controlled by the LMPC policy~\eqref{eq:policy}. Let Assumptions~\ref{ass:costStrictCovexity}--\ref{ass:stored} hold, then the problem~\eqref{eq:FTOCP} is feasible for all $t \geq 0$ and at every iteration $j\geq1$.
Moreover, the origin is asymptotically stable for the closed-loop system~\eqref{eq:sys} and~\eqref{eq:policy} at every iteration $j\geq1$.
\end{proposition}


\begin{proposition}[Theorem~2 in \cite{LMPC_linear}]
\label{prop2}
Consider system~\eqref{eq:sys} controlled by the LMPC policy~\eqref{eq:policy} and let Assumptions~\ref{ass:costStrictCovexity}--\ref{ass:stored} hold. Then the iteration cost $J_{0 \rightarrow \infty}^{j}(\cdot)$ does not increase with the iteration index $j$ and for all 
$t \geq 0$ we have that
\begin{equation*}
    J_{0\rightarrow \infty}^{j-1}(x_S) \geq \sum_{k=0}^{t-1}h(x_k^j,u_k^j)+ J_{t\rightarrow t+N}^{\scalebox{0.4}{LMPC},j}(x_t^j) \geq J_{0\rightarrow \infty}^{j}(x_S)
\end{equation*}
and $J_{t\rightarrow t+N}^{\scalebox{0.4}{LMPC},j}(x_t^j) \geq h(x_t^j, u_t^j) + J_{t+1 \rightarrow t+1+N}^{\scalebox{0.4}{LMPC},j}(x_{t+1}^j)$.
\end{proposition}


\begin{remark}
Notice that Assumption~\ref{ass:stored} and Propositions~\ref{prop:recFeasibility}--\ref{prop2} imply that at each iteration $j$ we have $J_{0 \rightarrow \infty}^j(x_s) = \sum_{t=0}^\infty h(x_t^j, u_t^j) \leq J_{0 \rightarrow \infty}^0(x_s) < \infty$.
\end{remark}

\section{Optimality and Performance Improvement}

First, we introduce a sufficient condition which guarantees optimality of the LMPC policy at convergence. Afterwards, we show that this sufficient condition also implies a strict performance improvement at each iteration.

\subsection{Optimality}

Assume that, after a finite number of iterations $c$, the closed-loop system~\eqref{eq:sys} and~\eqref{eq:policy} converges to a fixed-point, i.e,
\begin{equation*}
\begin{aligned}
    \boldsymbol{x}^c =\boldsymbol{x}^{c+1} = \dots =: \boldsymbol{x}^\infty \text{ and } \boldsymbol{u}^c =\boldsymbol{u}^{c+1} = \dots =: \boldsymbol{u}^\infty.
\end{aligned}
\end{equation*}
In this case, we will simply say that the LMPC has converged to a fixed-point\footnote{Notice that a ``point'' for our algorithm is a pair of state-input trajectories associated with one iteration of the task. Convergence to a fixed-point is guaranteed by the monotonicity property in Proposition~\ref{prop2}.} $(\boldsymbol{x}^\infty,\boldsymbol{u}^\infty)$.

Now, consider the following finite-time optimal control problem closely related to problem~\eqref{eq:inf_OCP}:
\begin{equation}\label{eq:FiniteHorizon}
	\begin{aligned}
	\!P_{t\rightarrow t+T}^*( x_t^\infty, x_{t+T}^{\infty}) =\! \min_{u_{0},\ldots,u_{T-1}} ~ &\sum_{k=0}^{T-1}  h(x_{k},u_{k}) \\
	\text{subject to} \quad ~ & x_{k+1}= Ax_{k}+Bu_{k}, \\
   & x_{k} \in \mathcal{X}, ~ u_{k} \in \mathcal{U},~ \\
	&x_{0}=x_t^\infty, ~x_{T} = x_{t+T}^{\infty},\\
	& \forall k \in \{0, \ldots, T-1\},
	\end{aligned}
\end{equation}
where the running cost, the dynamic constraint, the state and input constraints are the same as in~\eqref{eq:FTOCP}. Compare problem~\eqref{eq:FiniteHorizon} with problem~\eqref{eq:FTOCP}. Problem~\eqref{eq:FiniteHorizon} uses a horizon $T$, possibly longer than the horizon $N$ of problem~\eqref{eq:FTOCP}. Moreover, the terminal set in problem~\eqref{eq:FiniteHorizon} is a subset of the terminal set in problem~\eqref{eq:FTOCP} if $j\geq c$. Therefore, if the optimal solution to problem~\eqref{eq:FTOCP}  with $T=N$ is feasible for problem~\eqref{eq:FiniteHorizon}, then it is also optimal. 

In what follows, we introduce a sufficient condition which guarantees that $\boldsymbol{x}^\infty_{t:t+T} = [x_t^\infty,\ldots, x_{t+T}^\infty]$ is optimal for problem~\eqref{eq:FiniteHorizon} for all $T\geq0$ and for all $t\geq0$. Compared to the sufficient condition presented in~\cite[Theorem~3]{LMPC}, our condition can be easily checked and it may be used to verify that the closed-loop trajectory $\boldsymbol{x}^\infty$ associated with a fixed-point $(\boldsymbol{x}^\infty, \boldsymbol{u}^\infty)$ of the LMPC algorithm is optimal for problem~\eqref{eq:inf_OCP}, as shown in Section~\ref{sec:results}.

First, we rewrite problem~\eqref{eq:FiniteHorizon} in compact form:
\begin{equation}\label{eq:FiniteHorizonCompact}
	\begin{aligned}
	P_{t\rightarrow t+T}^*( x_t^\infty, x_{t+T}^{\infty})=\min_{\boldsymbol{z}_{t:t+T-1}} \quad & \boldsymbol{z}_{t:t+T-1}^\top Q_{T-1} \boldsymbol{z}_{t:t+T-1} \\
	\text{subject to} ~\quad & \GeqT \boldsymbol{z}_{t:t+T-1} = b_{\textrm{eq}} \\
   & \FinT \boldsymbol{z}_{t:t+T-1} \leq b_\mathrm{in},
	\end{aligned},
\end{equation}
where the vector  $\boldsymbol{z}_{t:t+T-1}=\text{Vec}(x_{0}, u_{0},\ldots, x_{T-1}, u_{T-1})$.
The matrices $Q_{T-1}$, $\GeqT$, $\FinT$ and the vectors $b_\textrm{eq}$ and $b_\mathrm{in}$ and  are defined in the Appendix~\ref{sec:Appendix-Mat}.

Let
\begin{equation*}
    \boldsymbol{z}_{t:t+T-1}^*= \text{Vec}(x_{t}^*, u_{t}^*, \ldots, x_{t+T-1}^*, {u_{t+T-1}^*})
\end{equation*} be the optimal solution to problem~\eqref{eq:FiniteHorizonCompact}. We recall that there exists a sequence of multipliers $\boldsymbol{\lambda}^*_{t:t+T-1}$ and $\boldsymbol{\delta}_{t:t+T-1}^* = [\boldsymbol{\delta}_{t:t+T-1}^{*,\textrm{active}}, \boldsymbol{\delta}_{t:t+T-1}^{*,\textrm{inactive}}]$ so that the following KKT conditions are satisfied:
\begin{equation*}
\begin{aligned}
    \!&\!\begin{bmatrix*}[l]
    \GeqT \\
    \FinTAc \\
    \end{bmatrix*}^\top \!\! \begin{bmatrix} \boldsymbol{\lambda}^*_{t:t+T-1}\\ \boldsymbol{\delta}_{t:t+T-1}^{*,\textrm{active}}\\
    \end{bmatrix}\!\! =\! - 2 Q_{T-1} \boldsymbol{z}_{t:t+T-1}^*, &&\quad (\textrm{Stationarity})\\[6pt]
    \!&\begin{aligned}\GeqT \, \boldsymbol{z}_{t:t+T-1}^* = b_{\textrm{eq}}, \\
    \FinT \, \boldsymbol{z}_{t:t+T-1}^* \leq b_\mathrm{in},\end{aligned} &&\!\!\!\!\!\!\!\!\! (\textrm{Primal Feasibility})\\[6pt]
    &\boldsymbol{\delta}^*_{t:t+T-1} \geq 0, &&\!\!\!\!\! (\textrm{Dual Feasibility}) \\[6pt]
    &\begin{aligned}
        \delta_{k|t,i}^* \, F_{k|t,i}^\textrm{in} = 0, \,&\forall i \in \mathcal{A}_{k|t},  \, \forall k \in \mathcal{T}_{t}, \\ \delta_{k|t,i}^* \, F_{k|t,i}^\textrm{in} = 0, \, &\forall i \in \mathcal{I}_{k|t}, \; \, \forall k\in \mathcal{T}_{t},
    \end{aligned} &&\!\!\!\!\!\!\!\! (\textrm{Complementarity})    \end{aligned}
\end{equation*}
where $\mathcal{T}_t=\{t,\ldots,t+T-1\}$ and, at optimum, $\FinTAc$ collects the active inequality constraints of problem~\eqref{eq:FiniteHorizonCompact}, and $\mathcal{A}_{k|t}$ and $\mathcal{I}_{k|t}$ are the set of indices associated with active and inactive inequality constraints at time $k\in \{t,\ldots, t+T-1\}$.
\begin{assumption}[LICQ]\label{ass:LICQ}
For a given fixed prediction horizon $N$ and for all $t\geq1$, the LICQ condition holds for problem~\eqref{eq:FiniteHorizonCompact} defined for $T=N-1$. 
\end{assumption}

The above assumption is the key ingredient that will allow us to show strict performance improvement. More explicitly, it states that at any time $t\geq 1$, given the following optimal solution to~problem~\eqref{eq:FiniteHorizonCompact} with horizon $T=N-1$: 
\begin{equation}
    \begin{aligned}
        &\boldsymbol{z}^*_{t:t+N-2} =\text{Vec}(x_t^*,u_t^*, \ldots, u_{t+N-2}^*,x_{t+N-2}^*),
        \end{aligned}
\end{equation}
the gradients of the active inequality constraints and those of the equality constraints are linearly independent at $\boldsymbol{z}^*_{t:t+N-2}$, i.e., the following matrix is full row rank
\begin{equation}\label{eq:Mdef}
    \MatTm = \begin{bmatrix*}[l]
    \GeqNm \\
    \FinNAcm \\
    \end{bmatrix*},
\end{equation}
where $\FinNAcm$ collects the active inequality constraints.

\begin{remark}
Note that the set containment condition used to establish~\cite[Theorem~3]{LMPC} only holds when no input constraints are active, a rather strong restriction. Moreover, the same condition only holds when the associated one-step forward and one-step backward reachable sets are full-dimensional, which is also not required for Assumption~\ref{ass:LICQ} to hold, as shown in our numerical example in Section~\ref{sec:results}.
\end{remark}

When the LICQ condition from Assumption~\ref{ass:LICQ} is satisfied, the sequences
\begin{equation*}
    \boldsymbol{x}^\infty_{t:t+T}\! = \![x_t^\infty, \ldots, x_{t+T}^\infty] \text{ and  } \boldsymbol{u}^\infty_{t:t+T-1}\! =\! [u_t^\infty, \ldots, u_{t+T-1}^\infty] 
\end{equation*}
are optimal for problem~\eqref{eq:FiniteHorizon}, for all $t\geq0$ and all $T>0$, as stated in the following theorem.


\begin{theorem}\label{th:newCond}
Consider system~\eqref{eq:sys} in closed-loop with the LMPC policy~\eqref{eq:policy} defined for a horizon $N$. Let Assumptions~\ref{ass:costStrictCovexity}--\ref{ass:LICQ} hold and assume that after $c$ iterations the closed-loop system~\eqref{eq:sys} and~\eqref{eq:policy} converges to a fixed-point $({\boldsymbol{x}}^\infty,{\boldsymbol{u}}^\infty)$. Then $({\boldsymbol{x}}_{t:t+T}^\infty, {\boldsymbol{u}}_{t:t+T-1}^\infty)$ is the optimizer of the finite-horizon optimal control problem~\eqref{eq:FiniteHorizon} for all $t\geq0$ and for all $T>0$.
\end{theorem}
\begin{proof}
The proof can be found in the Appendix~\ref{sec:mainProof}.
\end{proof}


\begin{remark}
Theorem~\ref{th:newCond} shows that the state-input vector $\boldsymbol{z}^\infty_{t:t+N-1} = \text{Vec}(x_t^\infty, u_t^\infty, \ldots, x_{t+N-1}^\infty, u_{t+N-1}^\infty)$ is optimal for problem~\eqref{eq:FiniteHorizonCompact} with $T=N$. Therefore, given a fixed-point $(\boldsymbol{x}^\infty,\boldsymbol{u}^\infty)$, Assumption~\ref{ass:LICQ} can be easily verified by checking the rank of the matrix $\MatTm$ associated with $\boldsymbol{z}^\infty_{t:t+N-1}$ for all times $t\geq1$. 
\end{remark}

Theorem~\ref{th:newCond} implies that $(\boldsymbol{x}^\infty_{0:T}, \boldsymbol{u}^\infty_{0:T-1})$ is optimal for problem~\eqref{eq:FiniteHorizon} for all $T\geq0$.
Next, we show that $(\boldsymbol{x}^\infty, \boldsymbol{u}^\infty)$ is optimal for problem~\eqref{eq:inf_OCP}.

\begin{theorem}
\label{th:optInf}
Consider system~\eqref{eq:sys} in closed-loop with the LMPC policy~\eqref{eq:policy} defined for a horizon $N$. Let Assumptions~\ref{ass:costStrictCovexity}-\ref{ass:LICQ} hold and assume that after $c$ iterations the closed-loop system~\eqref{eq:sys} and~\eqref{eq:policy} converges to a fixed-point $({\boldsymbol{x}}^\infty, {\boldsymbol{u}}^\infty)$.
Then, $({\boldsymbol{x}}^\infty, {\boldsymbol{u}}^\infty)$ is the optimizer of the infinite-horizon optimal control problem~\eqref{eq:inf_OCP}.
\end{theorem}

\begin{proof}
First, we notice that 
\begin{equation}\label{eq:sequenceDifference}
\begin{aligned}
        &\lim_{T \rightarrow \infty} || J^*_{0\rightarrow \infty}(x_S) - P_{0\rightarrow T}^*( x_S, x_{t+T}^{\infty})||\\
        &\qquad =\lim_{T \rightarrow \infty} || J^*_{0\rightarrow \infty}(x_S) - P_{0\rightarrow T}^*( x_S, x_{t+T}^{*}) \\
        &\qquad\quad + P_{0\rightarrow T}^*( x_S, x_{t+T}^{*}) - P_{0\rightarrow T}^*( x_S, x_{t+T}^{\infty})||\\
        &\qquad \leq \lim_{T \rightarrow \infty} || J^*_{0\rightarrow \infty}(x_S) - P_{0\rightarrow T}^*( x_S, x_{t+T}^{*}) || \\
        &\qquad\quad + \lim_{T \rightarrow \infty} || P_{0\rightarrow T}^*( x_S, x_{t+T}^{*}) - P_{0\rightarrow T}^*( x_S, x_{t+T}^{\infty})||.
\end{aligned}
\end{equation}
By definition we have that $ J^*_{0\rightarrow \infty}(x_S) =\lim_{T \rightarrow \infty} P_{0\rightarrow T}^*( x_S, x_{t+T}^{*})$, zeroing the first term of the sum. Furthermore, we notice that problem~\eqref{eq:FiniteHorizon} is a parametric QP and therefore its value function is Lipschitz continuous within a compact set for some constant $L$. These facts allow us to rewrite~\eqref{eq:sequenceDifference} as
\begin{equation*}
\begin{aligned}
        &\lim_{T \rightarrow \infty} || J^*_{0\rightarrow \infty}(x_S) - P_{0\rightarrow T}^*( x_S, x_{t+T}^{\infty})||\\
        &\qquad\leq \lim_{T \rightarrow \infty} || P_{0\rightarrow T}^*( x_S, x_{t+T}^{*}) - P_{0\rightarrow T}^*( x_S, x_{t+T}^{\infty})||\\
        &\qquad\leq \lim_{T \rightarrow \infty} L || x_{t+T}^{*} - x_{t+T}^{\infty}||.
\end{aligned}
\end{equation*}
By asymptotic stability~(Proposition~\ref{prop:recFeasibility}), sequences $\{x_{t}^\infty\}_{t=0}^\infty$ and $\{u_{t}^\infty\}_{t=0}^\infty$ converge to zero, hence for any $\epsilon>0$, there exists $N$ and $M$ such that  $x_{t+T}^\infty\in\mathcal{B}(\frac{\epsilon}{2})~\forall T \geq N$ and  $x_{t+T}^*\in\mathcal{B}(\frac{\epsilon}{2})~\forall T\geq M$. Hence, for all $T\geq \max\{N, M\}$, 
\begin{align*}
    || x_{t+T}^{*} - x_{t+T}^{\infty}||\leq|| x_{t+T}^{*}-0||+|| x_{t+T}^{\infty}-0||\leq\epsilon\;.
\end{align*}
Notice that, as $\epsilon$ is arbitrarily small,
\begin{align*}
    \begin{aligned}
        &\lim_{T \rightarrow \infty} || J^*_{0\rightarrow \infty}(x_S) - P_{0\rightarrow T}^*( x_S, x_{t+T}^{\infty})||\\
        &\qquad\leq \lim_{T \rightarrow \infty} L || x_{t+T}^{*} - x_{t+T}^{\infty}||=0\;.
\end{aligned}
\end{align*}
From the above equation we have that 
\begin{equation*}
    J^*_{0\rightarrow \infty}(x_S) = \lim_{T \rightarrow \infty} P_{0\rightarrow T}^*( x_S, x_{t+T}^{\infty}).
\end{equation*}
Finally, from Theorem~\ref{th:newCond} we have that $({\boldsymbol{x}}_{0:T}^\infty, {\boldsymbol{u}}_{0:T-1}^\infty)$ is the optimizer to problem $P_{0\rightarrow T}^*( x_S, x_{t+T}^{\infty})$ for all $T \geq 0$, which together with the above equation implies that $ ({\boldsymbol{x}}^\infty, {\boldsymbol{u}}^\infty) = \lim_{T\rightarrow\infty}({\boldsymbol{x}}_{0:T}^\infty, {\boldsymbol{u}}_{0:T-1}^\infty)$ is the optimizer to the infinite-horizon optimal control problem~\eqref{eq:inf_OCP}.
\end{proof}

\subsection{Performance Improvement}
In this section, we show that the closed-loop performance $J_{0 \rightarrow \infty}^j(x_S)$ associated with the LMPC policy is strictly decreasing\footnote{This result is stronger than the one from Proposition~\ref{prop2}, where we showed that the iteration cost is non-increasing at each iteration.} at each iteration until the closed-loop trajectory converges to the optimal one from problem~\eqref{eq:inf_OCP}.

\begin{lemma}\label{lemma1}
Consider system~\eqref{eq:sys} in closed-loop with the LMPC policy~\eqref{eq:policy}. Let Assumptions~\ref{ass:costStrictCovexity} and \ref{ass:stored} hold. If two iterations attain the same cost, then the associated trajectories are also the same, i.e., if $J^{j-1}_{0\rightarrow \infty}(x_S) = J^j_{0\rightarrow \infty}(x_S)$, then ${\boldsymbol{x}}^{j-1} = {\boldsymbol{x}}^{j}$.
\end{lemma}

\begin{proof}
We proceed by induction. First, assume that at the $j$-th iteration there exists a time $t$ such that $x_k^{j-1}=x_k^j$ for all $k \in \{0,\ldots, t\}$. From Proposition~\ref{prop2} and the convergence of the cost we have that
\begin{equation*}
    J_{0\rightarrow \infty}^{j-1}(x_S) = \sum_{k=0}^{t-1}h(x_k^j,u_k^j)+ J_{t\rightarrow t+N}^{\scalebox{0.4}{LMPC},j}(x_t) = J_{0\rightarrow \infty}^{j}(x_S).
\end{equation*}
Therefore, the optimal cost at time $t$ is
\begin{equation}\label{eq:costEq}
    J_{t\rightarrow t+N}^{\scalebox{0.4}{LMPC},j}(x_t) = J_{0\rightarrow \infty}^{j-1}(x_S)-\sum_{k=0}^{t-1}h(x_k^j,u_k^j).
\end{equation}
By our induction assumption, at time $t$ of the $j$-th iteration the state $x_t^j=x_t^{j-1}$. Now notice that at time $t$ of iteration $j$ the LMPC cost associated with the feasible trajectory $[x_{t}^{j-1},\ldots, x_{t+N}^{j-1}]$ is
\begin{equation}\label{eq:ineq}
    \begin{aligned}
&\sum_{k=t}^{t+N-1}h(x_k^{j-1}, u_k^{j-1}) + V^{j-1}(x_{t+N}^{j-1})\\
&\qquad\leq \sum_{k=t}^{t+N-1}h(x_k^{j-1}, u_k^{j-1}) + \sum_{k=t+N}^{\infty}h(x_k^{j-1}, u_k^{j-1})\\
&\qquad=J_{0\rightarrow \infty}^{j-1}(x_S)-\sum_{k=0}^{t-1}h(x_k^{j-1},u_k^{j-1})
\\
&\qquad=J_{0\rightarrow \infty}^{j-1}(x_S)-\sum_{k=0}^{t-1}h(x_k^{j},u_k^{j}).
\end{aligned}\end{equation}
Equations~\eqref{eq:costEq}--\eqref{eq:ineq} together with Assumption~\ref{ass:costStrictCovexity} and convexity of $V^j(\cdot)$ imply that $[x_{t}^{j-1},\ldots, x_{t+N}^{j-1}]$ is optimal at time $t$ and therefore $x_{t+1}^j= x_{t+1|t}^{j,*} = x_{t+1}^{j-1} $. Finally, notice that at time $t= 0$ we have that the initial state $x_0^{j-1} = x_0^j = x_S$, therefore we conclude by induction that $x_t^{j-1} = x_t^j$ for all $t \geq 0$.
\end{proof}

\begin{corollary}\label{th3}
Consider system~\eqref{eq:sys} in closed-loop with the LMPC policy~\eqref{eq:policy}. If Assumptions~\ref{ass:costStrictCovexity}--\ref{ass:LICQ} hold, then at iteration $j$ either one of the following must hold:
\begin{itemize}
    \item The closed-loop cost strictly decreases, i.e., $J_{0 \rightarrow \infty}^j(x_S)<J_{0 \rightarrow \infty}^{j-1}(x_S)$;
    \item The closed-loop cost is optimal for problem~\eqref{eq:inf_OCP}, i.e. $J_{0 \rightarrow \infty}^j(x_S)= J_{0 \rightarrow \infty}^{j-1}(x_S)=J^*_{0 \rightarrow \infty} (x_S)$.
\end{itemize}
\end{corollary}

\begin{proof}
From Proposition~\ref{prop2} we have that the iteration cost in non-increasing. Therefore, at each iteration either  $J^j_{0\rightarrow \infty}(x_S) < J^{j-1}_{0\rightarrow \infty}(x_S)$ or  $J^j_{0\rightarrow \infty}(x_S) = J^{j-1}_{0\rightarrow \infty}(x_S)$. Now consider the latter case, if $J^{j}_{0\rightarrow \infty}(x_S) = J^{j-1}_{0\rightarrow \infty}(x_S) $
then due to Lemma~\ref{lemma1} we have that ${\boldsymbol{x}^{j-1}} = {\boldsymbol{x}^{j}}$ which in turn implies $\mathcal{CS}^{j-1} = \mathcal{CS}^{j}$ and ${V}^{j-1} = {V}^{j}$. Consequently, ${\boldsymbol{x}^{j-1}} = {\boldsymbol{x}^{j}} = {\boldsymbol{x}^{j+1}} = \dots = {\boldsymbol{x}^{\infty}}$. Finally, Theorem~\ref{th:optInf} ensures that 
$J^{j-1}_{0\rightarrow \infty}(x_S) = \ldots = J^{\infty}_{0\rightarrow \infty}(x_S) = J^*_{0\rightarrow \infty}(x_S)$,
concluding the proof.
\end{proof}

\vspace{-0.4cm}

\section{Enlarging the Region of Attraction}
In this section, we present an iterative strategy that may be used to construct the safe set and the terminal cost when a first feasible trajectory is not given. This algorithm may also be utilized for enlarging the region of attraction associated with the controller when a first feasible trajectory is available.

First, we introduce the Region of Attraction (RoA) associated with the policy~\eqref{eq:policy}:
\begin{equation*}
\begin{aligned}
    \mathcal{C}^j=\{ x_0 \in \mathbb{R}^n~|& ~\exists \{\bar x_{k}\}_{k=0}^N, \{\bar u_{k}\}_{k=0}^{N-1} \text{ such that } \bar x_0= x_0,\\
    & \bar x_{k+1} = A\bar x_k + B \bar u_k, \bar x_k \in \mathcal{X}, \bar u_k \in \mathcal{U},\\
    &\forall k \in \{0,\ldots, N-1\}\text{ and } \bar x_{N} \in \mathcal{CS}^{j-1} \}.
\end{aligned}
\end{equation*}
By definition, for any initial condition $x_0 \in \mathcal{C}^j$ the closed-loop system~\eqref{eq:sys} and~\eqref{eq:policy} is asymptotically stable and the constraints from~\eqref{eq:constraints} are satisfied. Let 
    $\{v_1^j, \ldots, v_{n^j}^j\}$
be the $n^j$ vertices associated with the region of attraction $\mathcal{C}^j$. We notice that each vertex $v_i^j \in \mathcal{C}^j$ may be outside the the convex safe set $\mathcal{CS}^{j-1}$. Therefore, we can run a closed-loop simulation from each vertex $v_i^j$ to enlarge the convex safe set and, as a result, the region of attraction associated with the controller.
In Algorithm~\ref{algo:enlarge} we run closed-loop simulations from the vertices of $\mathcal{C}^j$, and we use the simulated data to  update the convex safe set $\mathcal{CS}^j$ and the $V$-function $V^j$. In Section~\ref{sec:resEnlarge}, we show that when a first feasible trajectory is not given, it is possible to set $\mathcal{CS}^0 =\{0\}$ and leverage Algorithm~\ref{algo:enlarge} to iteratively enlarge the region of attraction.

\floatname{algorithm}{Algorithm}
\begin{algorithm}[b]
  \caption{Domain Enlargement}\label{algo:enlarge}
  \begin{algorithmic}[1]
    \State \textbf{Input: } $M$
    \State Set $j=0$, $\mathcal{CS}^0 =\{0\}$, $V^0(x) = 0, \forall x \in \mathcal{CS}^0$
    \For{$i \in \{1, \ldots, M\}$} 
        \State Define the policy $\pi^j$ from~\eqref{eq:policy} using $\mathcal{CS}^{j-1}$ and $V^{j-1}$
        \State Compute the vertices $ \{v_1^j, \ldots, v_{n^j}^j\}$ of the set $\mathcal{C}^j$
        \For{$v \in \{v_1^j, \ldots, v_{n^j}^j\}$}
        \State Run a simulation setting $v$ as initial condition
    \State Construct $\mathcal{CS}^j$ and $V^j$ using data up to iteration $j$
    \State Set $j = j+1$
    \EndFor
    \EndFor
    \State \textbf{Outputs} $\mathcal{CS}^j$, $V^j$
  \end{algorithmic}
\end{algorithm}

\begin{remark}\label{rem:app} Computing the vertices of the region of attraction~$\mathcal{C}^j$ may be challenging. A computationally cheaper alternative is to replace the vertices $\{v_1^j, \ldots, v_{n^j}^j\}$ used in Algorithm~\ref{algo:enlarge} with the following points:
\begin{equation}\label{eq:optEnlarge}
\begin{aligned}
    \bar v_i^{j} = \argmax_{v\in\mathbb{R}^n} \quad & v^\top d_i \\
    \text{subject to} \quad & v \in \mathcal{C}^{j}, ~v^\top d_i^\perp = 0,
\end{aligned}
\end{equation}
where the set of user-specified vectors $\mathcal{D} = \{d_1, \ldots, d_n\}$ characterize the directions in which we wish to enlarge the RoA. In the above convex optimization problem $d_i^\perp \in \mathbb{R}^n$ denotes a vector orthogonal to $d_i \in \mathbb{R}^n$.
\end{remark}

\section{Results}\label{sec:results}
In this section, we test the LMPC
\footnote{Code is available at~{\tt \url{https://github.com/urosolia/LMPC_SimpleExample}.}} 
on the following infinite-time constrained linear quadratic regulator problem:
\begin{equation}\label{eq:infHorEx}
    \begin{aligned}
        J^*(x_S) = \min_{u_0, u_1, \ldots} \quad & \sum_{k=0}^\infty x_k^\top Q x_k + u_k^\top R u_k \\
        \text{subject to} \quad ~~ & x_{k+1} = \begin{bmatrix} 1 & 1 \\0 &1\end{bmatrix}x_k + \begin{bmatrix} 0 \\ 1 \end{bmatrix} u_k,~\forall k \geq 0, \\
        & ||x_k||_\infty \leq 15, || u_k ||\leq u_{\textrm{max}},~\forall k \geq 0,\\
        & x_0 = x_S,
    \end{aligned}
\end{equation}
where $Q=\textrm{diag}(1,1)$, $R=1$, and $x_S = [-14, 2]^\top$. 
In Section~\ref{sec:iterativeImp}, we initialize the LMPC algorithm with a suboptimal trajectory and we perform the regulation task until the closed-loop system converges to a fixed-point. Afterwards, we check if the LICQ condition from Assumption~\ref{ass:LICQ} is satisfied and we compare the steady-state solution with the optimal trajectory\footnote{The optimal trajectory is approximated by solving a finite time optimal control problem with a horizon $N=300$.}. In all tested scenarios, the LMPC converged in less than $20$ iterations. Finally, in Section~\ref{sec:resEnlarge} we do not assume that a first feasible trajectory is given, and we leverage Algorithm~\ref{algo:enlarge} to iteratively enlarge the region of attraction associated with the LMPC.

\begin{figure}[b]
\centering
\centering
\begin{tikzpicture}
\begin{axis}[
    xlabel=$x_1$,
    ylabel=$x_2$,
    ymin=0, 
    ymax=4.5,
    legend pos=south west,
    legend cell align=left,
    mark size=3.0pt,
    ]
  \pgfplotstableread{tikz/plot1_x_feasible.dat}{\dat}
  \addplot+ table [x={t}, y={y}] {\dat};
  \addlegendentry{First feasible trajectory}

  \pgfplotstableread{tikz/plot1_x_LMPC_N_3.dat}{\dat}
  \addplot+ table [x={t}, y={y}] {\dat};
  \addlegendentry{LMPC with $N=3$}

  \pgfplotstableread{tikz/plot1_x_LMPC_N_4.dat}{\dat}
  \addplot+ table [x={t}, y={y}] {\dat};
  \addlegendentry{LMPC with $N=4$}

  \pgfplotstableread{tikz/plot1_x_opt.dat}{\dat}
  \addplot+ table [x={t}, y={y}] {\dat};
  \addlegendentry{Optimal trajectory}
\end{axis}
\end{tikzpicture}
\caption{The figure shows the first feasible trajectory used to initialize the LMPC, the closed-loop trajectory at convergence, and the optimal solution. }\label{fig:ex1:closedLoop}
\end{figure}
\subsection{Iterative Improvement}\label{sec:iterativeImp}
First, we set $u_\textrm{max}=2$ and we synthesize the LMPC policy for $N=3$. At convergence the LICQ was satisfied for all $t\geq1$ and therefore the LMPC converged to the unique optimal solution to problem~\eqref{eq:infHorEx}. This fact is confirmed by Figure~\ref{fig:ex1:closedLoop}, where we reported the first feasible trajectory used to initialize the LMPC, the closed-loop trajectory at convergence, and the optimal solution. 
Notice that the closed-loop trajectory at convergence overlaps with the optimal solution to problem~\eqref{eq:infHorEx}.


In the second example, we set $u_\textrm{max}=1.5$ and we synthesize the LMPC policy for $N=3$ and $N=4$. Figure~\ref{fig:ex2:closedLoop} shows that for $N=4$ the closed-loop trajectory converges to the optimal solution to problem~\eqref{eq:infHorEx}. However, for $N=3$ the closed-loop system does not converge to the optimal closed-loop behavior. These results are in agreement with Theorem~\ref{th:newCond}. Indeed, Assumption~\ref{ass:LICQ} is satisfied for $N=4$, but it is not satisfied for $N=3$.
Tables~\ref{table:KKT_t0}--\ref{table:KKT_t1} show the KKT multipliers for problem~\eqref{eq:FiniteHorizonCompact} with $N=4$ at times $t=0$ and $t=1$. We notice that, as shown in the proof of the Theorem~\ref{th:newCond} (in particular in equation~\eqref{eq:KKT_equality}), we have that $\lambda_{k|0}^* = \lambda_{k|1}^*$ for all $k \in \{1,\ldots,4\}$ and $\delta^{*,\textrm{active}}_{1|0,0} = \delta^{*,\textrm{active}}_{1|1,0}$. 

\noindent
\begin{figure}[t]
\centering
\begin{tikzpicture}
\begin{axis}[
    xlabel=$x_1$,
    ylabel=$x_2$,
    ymin=0, 
    ymax=4.5,
    legend pos=south west,
    legend cell align=left,
    mark size=3.0pt,
    ]
  \pgfplotstableread{tikz/plot2_x_feasible.dat}{\dat}
  \addplot+ table [x={t}, y={y}] {\dat};
  \addlegendentry{First feasible trajectory}
 
  \pgfplotstableread{tikz/plot2_x_LMPC_N_3.dat}{\dat}
  \addplot+ table [x={t}, y={y}] {\dat};
  \addlegendentry{LMPC with $N=3$}

  \pgfplotstableread{tikz/plot2_x_LMPC_N_4.dat}{\dat}
  \addplot+ table [x={t}, y={y}] {\dat};
  \addlegendentry{LMPC with $N=4$}

  \pgfplotstableread{tikz/plot2_x_opt.dat}{\dat}
  \addplot+ table [x={t}, y={y}] {\dat};
  \addlegendentry{Optimal trajectory}
\end{axis}
\end{tikzpicture}
\vspace{-0.4cm}
\caption{The figure shows the first feasible trajectory used to initialize the LMPC, the optimal solution, and the closed-loop trajectory at convergence for $N=3$ and $N=4$. We notice that for $N=4$ the LMPC converges to the optimal behavior. }\label{fig:ex2:closedLoop}
\end{figure}
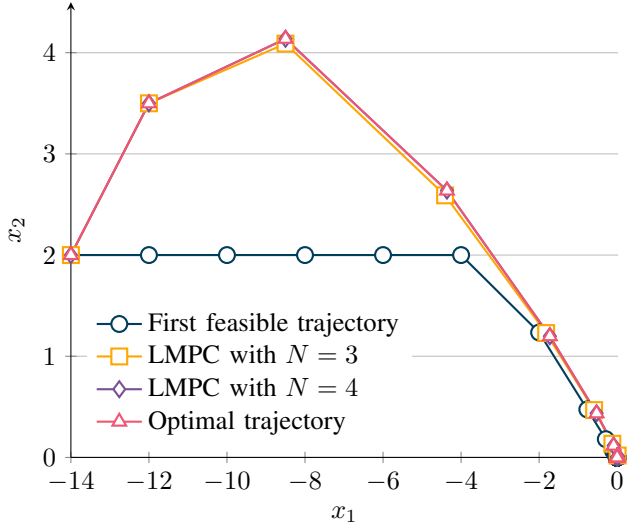


\begin{table}[h!]
\centering
\setlength\tabcolsep{2.5pt}
\begin{tabular}{ccccccc} \midrule
$\lambda_{0|0}^*$& $\lambda_{1|0}^*$& $\lambda_{2|0}^*$ & $\lambda_{3|0}^*$ &$\lambda_{4|0}^*$ &$\delta_{0|0,0}^{*, \textrm{active}}$&$\delta_{1|0,0}^{*, \textrm{active}}$\\ \midrule
$\begin{bmatrix} 82.21 \\ 74.70 \end{bmatrix}$        & $\begin{bmatrix} 54.21 \\ 24.49 \end{bmatrix}$         & $\begin{bmatrix}30.21 \\ 1.27 \end{bmatrix}$        & $\begin{bmatrix} 13.21 \\ -3.66 \end{bmatrix}$ & $\begin{bmatrix} 4.49 \\ -2.87 \end{bmatrix}$ & 21.49 & 0.66 \\ \midrule
\end{tabular}
\caption{KKT multipliers for problem~\eqref{eq:FiniteHorizonCompact} at time $t=0$ and for $N=4$}
\label{table:KKT_t0}
\end{table}
\vspace{-0.7cm}

\begin{table}[h!]
\centering
\setlength\tabcolsep{2.5pt}
\begin{tabular}{cccccc} \midrule
$\lambda_{1|1}^*$& $\lambda_{2|1}^*$& $\lambda_{3|1}^*$ & $\lambda_{4|1}^*$ &$\lambda_{5|1}^*$ &$\delta_{1|1,0}^{*, \textrm{active}}$ \\ \midrule
$\begin{bmatrix} 54.21 \\ 24.49 \end{bmatrix}$         & $\begin{bmatrix}30.21 \\ 1.27 \end{bmatrix}$        & $\begin{bmatrix} 13.21 \\ -3.66 \end{bmatrix}$ & $\begin{bmatrix} 4.49 \\ -2.87 \end{bmatrix}$ & $\begin{bmatrix} 1.04 \\ -1.52 \end{bmatrix}$ & 0.66 \\ \midrule
\end{tabular}
\caption{KKT multipliers for problem~\eqref{eq:FiniteHorizonCompact} at time $t=1$ and for $N=4$}
\label{table:KKT_t1}
\end{table}

\vspace{-0.4cm}

Finally, in Figure~\ref{fig:ex1:Input} we compare the LMPC input sequences at convergence for $N=4$ and $N=3$. 
At time $t=2$ the input is saturated and therefore the matrix $M_{t:t+2} \in \mathbb{R}^{7 \times 6}$ for $t\in \{1,2\}$ is not full row rank and the LICQ condition for $N=3$ is not fulfilled.
For this reason, the LMPC with horizon $N=3$ cannot explore the state space and the algorithm converges to a fixed-point which is not optimal.
Even though the solution is suboptimal with a prediction horizon equal to three, the sacrifice of optimality is subtle as the LICQ is satisfied for all $t \geq 3$. Specifically, when the prediction horizon $N=3$, the trajectory starting from $x_2^\infty$ is optimal, i.e., $J^*_{2\rightarrow \infty}(x_2^\infty)=J^\infty_{2\rightarrow \infty}(x_2^\infty)$.

\noindent
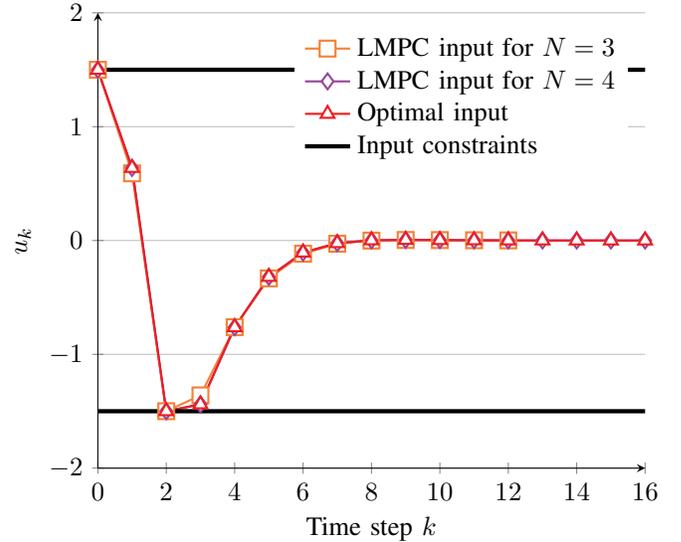
\begin{figure}
\centering
\begin{tikzpicture}
\begin{axis}[
    xlabel=Time step $k$,
    ylabel=$u_k$,
    ymin=-2,  
    ymax=2,
    xmin=0,
    xmax=16,
    legend pos=north east,
    legend cell align=left,
    mark size=3.0pt,
    ]
  \pgfplotstableread{tikz/plot3_u_LMPC_N_3.dat}{\dat}
  \addplot+ [mark=square*,mark options={fill=white}, Orange] table [x={t}, y={y}] {\dat};
  \addlegendentry{LMPC input for $N=3$}

  \pgfplotstableread{tikz/plot3_u_LMPC_N_4.dat}{\dat}
  \addplot+ [mark=diamond*,mark options={fill=white}, Purple] table [x={t}, y={y}] {\dat};
  \addlegendentry{LMPC input for $N=4$}

  \pgfplotstableread{tikz/plot3_u_opt.dat}{\dat}
  \addplot+ [mark=triangle*,mark options={fill=white}, Red] table [x={t}, y={y}] {\dat};
  \addlegendentry{Optimal input}

  \addplot [domain=0:16, black, ultra thick] {1.5};
  \addplot [domain=0:16, black, ultra thick] {-1.5};
  \addlegendentry{Input constraints}
\end{axis}
\end{tikzpicture}
\vspace{-0.5cm}
\caption{The figure shows the LMPC input sequences at convergence for $N=3$, $N=4$, and the optimal input sequence.}\label{fig:ex1:Input}
\end{figure}

\vspace{-0.6cm}

\subsection{Domain Enlargement}\label{sec:resEnlarge}
In this section, we initialize $\mathcal{CS}^0 = \{0\}$ and we leverage Algorithm~\ref{algo:enlarge} to iteratively enlarge the convex safe set $\mathcal{CS}^j$ and to construct the $V$-function. We test both Algorithm~\ref{algo:enlarge}, where the vertices of $\mathcal{C}^j$ are computed exactly, and a computationally cheaper alternative where 
instead of the vertices of $\mathcal{C}^j$ 
we used the points computed via problem~\eqref{eq:optEnlarge}
for $\mathcal{D}=\{ \pm [1, -0.3]^\top, \pm [1, -0.5]^\top, \pm [1, -0.7]^\top\}$. The set $\mathcal{D}$ is chosen to enlarge the RoA in the second and fourth quadrant.
Furthermore, we compare the region of attraction resulting from our approaches with an MPC controller, where the terminal constraint set is the maximal invariant set $\mathcal{O}_\infty$ associated with the LQR controller and the terminal cost is $x^\top P x$ for the matrix $P$ given by the solution of the discrete time Riccati equation. Figure~\ref{eq:regComp} shows the region of attractions for our strategies and the standard MPC approach. We notice that the region of attraction associated with the LMPC is larger.  This result is expected, as in the LMPC methodology we used a control invariant set as terminal constraint; whereas in the standard MPC approach we used the maximal control invariant associated with the LQR controller as terminal constraint and, as discussed in~\cite{grieder2004computation}, the size of the region of attraction is affected by the horizon length.

\begin{figure}
\centering
\begin{tikzpicture}
    \begin{axis}[
    xlabel=$x_1$,
    ylabel=$x_2$,
    ymin=-10,  
    ymax=15.5,
    xmin=-15.1,
    xmax=15.1,
    legend columns=2,
    legend pos=north east,
    legend cell align=left,
    mark size=3.0pt,
    ]
  \pgfplotstableread{tikz/Oinf_data.dat}{\dat}
  \addplot+ [Maroon] table {\dat};
  \addlegendentry{$\mathcal{O}_\infty$}

    \pgfplotstableread{tikz/K3LMPC_S2_data.dat}{\dat}
    \addplot+ [mark=triangle*,mark options={fill=white}, BurntOrange] table {\dat};
    \addlegendentry{LMPC RoA w/ Approx.}
  
    \pgfplotstableread{tikz/K3MPC_data.dat}{\dat}
    \addplot+ [NavyBlue] table {\dat};
    \addlegendentry{MPC RoA}
  
    \pgfplotstableread{tikz/K3LMPC_S1_data.dat}{\dat}
    \addplot+ [mark=square*,mark options={fill=white}, LimeGreen] table {\dat};
    \addlegendentry{LMPC RoA w/ Algo. 1}
  
  \pgfplotstableread{tikz/Cinf_data.dat}{\dat}
    \addplot+ [mark=*,mark options={fill=white, solid}, dashed, Black] table {\dat};
    \addlegendentry{$\mathcal{C}_\infty$}

\end{axis}
\end{tikzpicture}
\vspace{-0.9cm}
\caption{Comparison between the Regions of Attraction (RoA) associated with the MPC and the LMPC policy updated using Algorithm~\ref{algo:enlarge} (LMPC RoA w/ Algo. 1) and Algorithm~\ref{algo:enlarge} with the approximation described in Remark~\ref{rem:app} (LMPC RoA w/ Approx.). The figure shows also the set of stabilizable states $\mathcal{C}_\infty$. Notice that  $\mathcal{C}_\infty$ overlaps with the region of attraction associated with the proposed strategy, when Algorithm~\ref{algo:enlarge} is used for enlarging the LMPC~domain.}\label{eq:regComp}
\end{figure}


\vspace{-0.2cm}

\section{Conclusions}
In this paper, we presented a sufficient LICQ condition which guarantees strict performance improvement and optimality of the LMPC scheme. Compared to our previous work this condition can be easily checked and it holds for a larger class of systems. Furthermore, as demonstrated in simulations, our result can be used to adaptively select the LMPC prediction horizon to guarantee optimality of the closed-loop system at convergence.

\vspace{-0.8cm}

\section{Appendix}\label{sec:Appendix}
\subsection{Matrices from problem~\eqref{eq:FiniteHorizonCompact}}\label{sec:Appendix-Mat}
By definition we have that $Q_{N-1} = I_{N-1} \otimes \text{diag}(Q,R)$, $b_\textrm{eq} = \text{Vec}(x_t^\infty, 0, \ldots, 0, x_{t+T}^\infty)$, $b_{\textrm{in}} = \mathds{1}_{N-1} \otimes [b_x^\top, b_u^\top]^\top$, 
\begin{equation}\label{eq:Gpatter}
\begin{aligned}
     (\GeqN)^\top &= \begin{bmatrix*}[r]
     I_{n} &     -A^\top&          &    &     &    &        \\ 
       &     -B^\top&          &    &     &    &        \\
       &      I_{n}&   -A^\top&    &     &    &        \\
       &            &   -B^\top&    &     &    &        \\
       &            &          &    \ddots&    &        \\
       &            &          &          &   I_{n}& -A^\top\\
       &            &          &          &    & -B^\top\\
      \end{bmatrix*},\\
      (\FinN)^\top &= 
    I_{N-1} \otimes \text{diag}( F_{x}^\top,  F_{u}^\top).
\end{aligned}
\end{equation}
and
\begin{equation}\label{eq:Fpatter}
    \MatT^\top = \begin{bmatrix}(\GeqN)^\top & (\FinN)^\top \end{bmatrix}
\end{equation}

\vspace{-0.4cm}

\subsection{Proof of Theorem~\ref{th:newCond}}\label{sec:mainProof}
We start with a proof outline. First, we show that for all $t\in \{0, 1, \ldots \}$ the state-input sequences
\begin{equation}\label{eq:optimizerTime_t}
\begin{aligned}
    [x_t^\infty, \ldots, x_{t+N}^\infty] \text{ and } [u_t^\infty, \ldots, u_{t+N-1}^\infty]
\end{aligned}
\end{equation}
are the optimizer to problem $ J^{\scalebox{0.5}{LMPC},j}_{t\rightarrow t+N}(x_t^\infty)$ for $j \geq c$. Leveraging this result, we will show that the $N$-steps trajectories starting at time $t$ and $t+1$ can be combined to show optimality the $(N+1)$-steps trajectory $[x_t^\infty, \ldots, x_{t+N+1}^\infty]$. Thus, by induction we show optimality over an infinite horizon.

Recall that from Proposition~\ref{prop2} we have
\begin{equation}\label{eq:cost_bound}
\begin{aligned}
        J^{\scalebox{0.5}{LMPC},j}_{t\rightarrow t+N}(x_t^{j}) &\geq \sum_{k=t}^{\infty}h(x_k^j,u_k^j)+ \lim_{k\rightarrow \infty}J_{k\rightarrow k+N}^{\scalebox{0.4}{LMPC},j}(x_k^j).
\end{aligned}
\end{equation}
Furthermore, by definition~\eqref{eq:Qfun} the $V$-function is a lower-bound to the closed-loop cost, i.e., $V^j(x_t^j) \leq \sum_{k=t}^\infty h(x_k^j, u_k^j)$. Now notice that equation~\eqref{eq:cost_bound}, together with definition~\eqref{eq:Qfun} and the stability of the closed-loop system, implies that
\begin{equation*}
\begin{aligned}
    J^{\scalebox{0.5}{LMPC},c+1}_{t\rightarrow t+N}(&x_t^{c+1}) \geq \sum_{k=t}^{\infty}h(x_k^j,u_k^j) \\&= \sum_{k=t}^{t+N-1} h(x_k^{c+1}, u_k^{c+1}) + \sum_{k=t+N}^{\infty} h(x_k^{c+1}, u_k^{c+1})\\
    &\geq\sum_{k=t}^{t+N-1} h(x_k^{c+1}, u_k^{c+1}) + V^{c+1}(x_{t+N}^{c+1}).
\end{aligned}
\end{equation*}
We notice that by assumption $x_t^j=x^\infty_t$ for all iterations $ j \geq c$, therefore $V^{c+1}(x_{t+N}^{c+1})= V^{c}(x_{t+N}^{\infty})$ and 
\begin{equation*}
    J^{\scalebox{0.5}{LMPC},c+1}_{t\rightarrow t+N}(x_t^{c+1}) \geq \sum_{k=t}^{t+N-1} h(x_t^{\infty}, u_t^{\infty}) + V^{c}(x_N^{\infty}).
\end{equation*}
Finally, we notice that~\eqref{eq:optimizerTime_t} is a feasible solution to $ J^{\scalebox{0.5}{LMPC},c+1}_{t\rightarrow t+N}(x_t^{c+1})$ and achieves the above lower bound. Therefore, the positive definiteness of $R$  from Assumption~\ref{ass:costStrictCovexity} implies that the state-input sequences in~\eqref{eq:optimizerTime_t} are the unique optimal solution to problem $ J^{\scalebox{0.5}{LMPC},c+1}_{t\rightarrow t+N}(x_t^{c+1})$ for all $t\in \{0,1,\ldots\}$.

Next, we show that the following sequences of $N+1$ states and $N$ inputs $[x_0^\infty, \ldots, x_{N+1}^\infty] \text{ and } [u_0^\infty, \ldots, u_{N}^\infty]$
are optimal for problem~\eqref{eq:FiniteHorizon} defined for a horizon of $N+1$ steps. Due to Assumption~\ref{ass:LICQ}, the optimizer is a KKT point~\cite[p.~21]{fiacco1990nonlinear} and is uniquely defined by its multipliers. The following analysis is therefore built on the KKT system. 

As discussed, the state-input sequences in~\eqref{eq:optimizerTime_t} are optimal for problem $ J^{\scalebox{0.5}{LMPC},c+1}_{t\rightarrow t+N}(x_t^{c+1})$, for all $t\in \{0,1,\ldots\}$. Therefore, the vector $\boldsymbol{z}_{t:t+N-1}^\infty = \text{Vec}(x_t^\infty,u_t^\infty,\ldots, x_{t+N-1}^\infty, u_{t+N-1}^\infty)$
is optimal for problem $P_{t\rightarrow t+N}^*( x_{t}^\infty, x_{t+N}^{\infty})$ from ~\eqref{eq:FiniteHorizonCompact}
and there exists a set of multipliers such that the following necessary optimality condition is satisfied:
\begin{equation}\label{eq:Stationarity}
    \MatT^\top \begin{bmatrix*}[l] \LagEq_{t:t+N-1|t} \\ \LagIn_{t:t+N-1|t} \end{bmatrix*} = -2 Q_{N-1}  \boldsymbol{z}^\infty_{t:t+N-1} .
\end{equation}
In the above equation, the matrix $\MatT^\top$ collects the gradient of the active constraints as defined in~\eqref{eq:Mdef}. The vector $\LagEq_{t+N-1|t} = \text{Vec}( \lambda_{t|t}^*, \ldots, \lambda_{t+N-1|t}^*)$ collects the KKT multipliers associated with the equality constraints. 
Moreover, the KKT multipliers associated with the active inequality constraints at time $k$ are stacked into the vector $\LagIn_{k|t}$,  whose entries are $\delta^{*,\textrm{active}}_{k|t,i}$ for each index $i$ from the index set $\mathcal{A}_{k|t}$ associated with the active inequality constraints. 
Similarly, the index set for the inactive constraints and their multipliers are denoted by $\mathcal{I}_{k|t}$ and $\delta^{*,\textrm{inactive}}_{k|t}$.

Now, we construct the KKT multipliers for problem $P_{t+1\rightarrow t+N}^*( x_{t+1}^\infty, x_{t+N}^{\infty})$ defined over a horizon of $N-1$ steps. Notice that the matrix $M^\top_{t+1:t+N-1}$ is obtained from the matrix $M^\top_{t:t+N-1}$ by removing the rows and columns related to
the equality constraints and the active inequality constraints at time $t$ in problem $P_{t\rightarrow t+N}^*( x_t^\infty, x_{t+N}^{\infty})$ defined over $N$ steps. 
Therefore, by definitions~\eqref{eq:Gpatter}--\eqref{eq:Fpatter} and equation~\eqref{eq:Stationarity} we have
\begin{equation}\label{eq:OptimizerShrinkStart}
        \MatTmp^\top \begin{bmatrix*}[l] \LagEq_{t+1:t+N-1|t} \\ \LagIn_{t+1:t+N-1|t} \end{bmatrix*} = -2 Q_{N-2}  \boldsymbol{z}^\infty_{t+1:t+N-1} ,
\end{equation}
where the cost matrix $Q_{N-2} = I_{N-2} \otimes \text{diag}(Q, R)$, $\LagEq_{t+1:t+N-1|t} = \text{Vec}(\lambda^*_{t+1|t}, \ldots, \lambda^*_{t+N-1|t})$ and  $\LagIn_{t+1:t+N-1|t}$ collects the multipliers $\delta^{*,\mathrm{active}}_{k|t, i}$, for all $i \in \mathcal{A}_{k|t}$ and for all $k\in\{t+1,\ldots, t+N-1\}$. Equation~\eqref{eq:OptimizerShrinkStart} implies that a subset of the KKT multipliers from~\eqref{eq:Stationarity}, which are associated with problem~$P_{t\rightarrow t+N}^*( x_t^\infty, x_{t+N}^{\infty})$ defined over a horizon of $N$ steps, 
can be used to show optimality of the vector 
\begin{equation}\label{eq:zOpt}
    \boldsymbol{z}_{t+1:t+N-1}^\infty = \text{Vec}( x_{t+1}^\infty,u_{t+1}^\infty,\ldots, x_{t+N-1}^\infty, u_{t+N-1}^\infty)
\end{equation} 
for problem $P_{t+1\rightarrow t+N}^*( x_{t+1}^\infty, x_{t+N}^{\infty})$ defined over a horizon of $N-1$ steps.

In the following we show optimality of  above vector $\boldsymbol{z}^\infty_{t+1:t+N-1}$ using the KKT multipliers associated with problem $P_{t+1\rightarrow t+N+1}^*( x_{t+1}^\infty, x_{t+N+1}^{\infty})$ defined over $N$ steps.
As discussed, the state-input sequences~\eqref{eq:optimizerTime_t} for $t=t+1$ are optimal for $J^{\scalebox{0.5}{LMPC},c+1}_{t+1\rightarrow N+1}(x_{t+1}^\infty)$ and consequently
$\boldsymbol{z}_{t+1:t+N}^\infty = \text{Vec}(x_{t+1}^\infty,u_{t+1}^\infty,\ldots, x_{t+N}^\infty, u_{t+N}^\infty)$
is optimal for problem $P_{t+1\rightarrow t+N+1}^*( x_{t+1}^\infty, x_{t+N+1}^{\infty})$. Therefore, there exists a set of multipliers such that the stationarity condition is satisfied: 
\begin{equation*}
    M_{t+1:t+N}^\top \begin{bmatrix*}[l] \LagEq_{t+1:t+N|t+1} \\ \LagIn_{t+1:t+N|t+1} \end{bmatrix*} = -2 Q_{N-1} \boldsymbol{z}^\infty_{t+1:t+N} .
\end{equation*}
The above condition implies that
\begin{equation}\label{eq:OptimizerShrinkEnd}
    M_{t+1:t+N-1}^\top \begin{bmatrix*}[l] \LagEq_{t+1:t+N-1|t+1} \\ \LagIn_{t+1:t+N-1|t+1} \end{bmatrix*} =-2 Q_{N-2} \boldsymbol{z}^\infty_{t+1:t+N-1}.
\end{equation}
Therefore, the vectors of KKT multipliers $\LagEq_{t+1:t+N-1|t+1}$ and $\LagIn_{t+1:t+N-1|t+1}$ can be used to show optimality of the state-input vector $\boldsymbol{z}^\infty_{t+1:t+N-1}$ from~\eqref{eq:zOpt} for problem $P_{t+1\rightarrow t+N+1}^*( x_{t+1}^\infty, x_{t+N+1}^{\infty})$ defined over $N-1$ steps.

Finally, we show that the KKT multipliers from~\eqref{eq:OptimizerShrinkStart} and~\eqref{eq:OptimizerShrinkEnd} can be combined to prove optimality of the $(N+1)$-steps trajectory $[x_t^\infty, \ldots, x_{t+N+1}^\infty]$. From Assumption~\ref{ass:LICQ}, we have that $M_{t+1:t+N-1}^\top$ is full column rank. Therefore, from equations~\eqref{eq:OptimizerShrinkStart} and~\eqref{eq:OptimizerShrinkEnd} we have that
\begin{equation}
\label{eq:KKT_equality}
         \!\begin{rcases}
     \mathcal{A}_{k|t}=\mathcal{A}_{k|t+1},~\mathcal{I}_{k|t}=\mathcal{I}_{k|t+1}\\
    {\lambda}_{k|t}^* = \lambda_{k|t+1}^*\\
    \delta_{k|t,i}^{*,\mathrm{active}}  = \delta_{k|t+1,i}^{*,\mathrm{active}}, \forall i \in \mathcal{A}_{k|t},\\
    \delta_{k|t,i}^{*,\mathrm{inactive}}  = \delta_{k|t+1,i}^{*,\mathrm{inactive}}, \forall i \in \mathcal{I}_{k|t}.\\
    \end{rcases}\!\forall k\!\in\!\{t\!+\!1,\dots,t\!+\!N\!-\!1\} \\.
\end{equation}
 The above equation implies that the KKT multipliers $\boldsymbol{\bar \lambda}_{t:t+N}$ and $\boldsymbol{\bar \delta}_{t:t+N} = [\boldsymbol{\bar \delta}^\textrm{active}_{t:t+N},\boldsymbol{\bar \delta}^\textrm{inactive}_{t:t+N}]$ with entries
\begin{equation*}\label{eq:KKT_multipliers}
\begin{aligned}
    &\bar{\lambda}_{t|t} = \lambda_{t|t}^*\\
    &\bar{\delta}_{t|t,i}^\mathrm{active}  = \delta_{t|t,i}^{*,\mathrm{active}},  \forall i \in \mathcal{A}_{t|t}\\
    &\bar{\delta}_{t|t,i}^\mathrm{inactive}  = \delta_{t|t,i}^{*,\mathrm{inactive}},  \forall i \in \mathcal{I}_{t|t}\\
    &\!\begin{rcases}
    \bar{\lambda}_{k|t} = \lambda_{k|t}^*\\
    \bar{\delta}_{k|t,i}^\mathrm{active}  = \delta_{k|t,i}^{*,\mathrm{active}}, \forall i \in \mathcal{A}_{k|t},\\
    \bar{\delta}_{k|t,i}^\mathrm{inactive}  = \delta_{k|t,i}^{*,\mathrm{inactive}}, \forall i \in \mathcal{I}_{k|t},\\
    \end{rcases}\!\forall k\in\{t\!+\!1,\dots,t\!+\!N\!-\!1\}\\
    &\bar{\lambda}_{t+N|t} = \lambda_{t+N|t+1}^*\\
    &\bar{\delta}_{t+N|t,i}^\mathrm{active}  = \delta_{t+N|t+1,i}^{*,\mathrm{active}},  \forall i \in \mathcal{A}_{t+N|t+1}\\
    &\bar{\delta}_{t+N|t,i}^\mathrm{inactive}  = \delta_{t+N|t+1,i}^{*,\mathrm{inactive}},  \forall i \in \mathcal{I}_{t+N|t+1}\\
    \end{aligned}
\end{equation*}
satisfy the stationarity condition
\begin{equation}\label{eq:stat}
    M_{t:t+N}^\top \begin{bmatrix*} \boldsymbol{\bar \lambda}_{t:t+N} \\
    \boldsymbol{\bar \delta}^{*,\mathrm{active}}_{t:t+N}\end{bmatrix*}    = 2Q_{N} \boldsymbol{z}_{t:t+N}^\infty.
\end{equation}
By definition $\boldsymbol{\bar \delta}_{t:t+N}$ is dual feasible and $(\boldsymbol{z}_{t:t+N}^\infty, \boldsymbol{\bar \delta}_{t:t+N})$ satisfies the complementarity conditions. Therefore, equation~\eqref{eq:stat} implies that the feasible state-input vector $\boldsymbol{z}_{t:t+N}^\infty = \text{Vec}(x_t^\infty, u_t^\infty, \ldots, x_{t+N}^\infty, u_{t+N}^\infty)$ is optimal for problem~\eqref{eq:FiniteHorizonCompact} at time $t$ with horizon $N+1$. In conclusion, we have that $(\boldsymbol{x}^\infty_{t:t+N+1}, \boldsymbol{u}^\infty_{t:t+N})$ is the optimizer to problem~\eqref{eq:FiniteHorizon}  with horizon $N+1$. 

Now, we notice that from Assumption~\ref{ass:LICQ} the matrix $M_{t:t+T}$ is full row rank for all $t\geq1$ and $T\geq N-1$. By induction, the above argument can be iterated to show that $\boldsymbol{z}_{t:t+T}^\infty$ is the optimal solution to problem~\eqref{eq:FiniteHorizonCompact} for all $t\geq0$ and for all $T\geq N$. Consequently, $(\boldsymbol{x}^\infty_{t:t+T}, \boldsymbol{u}^\infty_{t:t+T-1})$ is the optimizer to problem~\eqref{eq:FiniteHorizon} for all $t\geq0$ and for all $T\geq N$. Finally, by standard dynamic programming arguments  $(\boldsymbol{x}^\infty_{t:t+T}, \boldsymbol{u}^\infty_{t:t+T-1})$ is the optimizer to problem~\eqref{eq:FiniteHorizon} for all $t\geq0$ and for all $T < N$, which concludes the proof.

\bibliographystyle{IEEEtran}
\bibliography{IEEEabrv,mybib}

\end{document}